\documentclass[12pt]{amsart}
\usepackage{amsthm,amsmath,amssymb,mathrsfs}
\usepackage{enumerate}
\usepackage{color}

\makeatletter
\@namedef{subjclassname@2010}{%
  \textup{2010} Mathematics Subject Classification}
\makeatother
\usepackage{amsthm,amsmath,amssymb,mathrsfs}
\theoremstyle{plain}
\newtheorem{theorem}{Theorem} [section]
\newtheorem{corollary}[theorem]{Corollary}
\newtheorem{lemma}[theorem]{Lemma}
\newtheorem{proposition}[theorem]{Proposition}

\theoremstyle{definition}
\newtheorem{definition}[theorem]{Definition}

\newtheorem{standing}[theorem]{Standing Assumptions}
\newtheorem{remark}[theorem]{Remark}

\newtheorem{example}[theorem]{Example}

\def\C{{\mathbb{C}}}
\def\N{{\mathbb{N}}}
\def\R{{\mathbb{R}}}
\def\Tor{{\mathbb{T}}}
\def\Z{{\mathbb{Z}}}

\def\A{{\mathcal{A}}}

\def\clspan{{\overline{\mathrm{span}}}}

\newcommand{\p}{\varphi}
\newcommand{\w}{\omega}

\newcommand{\g}{\gamma}
\newcommand{\la}{\lambda}
\newcommand{\dd}{\delta}
\newcommand{\G}{\Gamma}

\newcommand{\Dd}{\Delta}
\newcommand{\T}{\mathcal{T}}

\newcommand{\sn}{\overline{\text{span}}}

\newtheorem{thm*}{Teorema}[section]

\newcommand{\CHI}{\hbox{\raise .4ex \hbox{$\chi$}}}

\def\Span{{\text{\rm span}}}  
\def\subset{\subseteq}
\def\supp{{\textrm{supp}}}

\def\clspan{{\overline{\mathrm{span}}}}

\def\T{{\mathcal{T}}}

\begin{document}

\baselineskip=17pt

\title[Shift-modulation invariant spaces on LCA groups]{Shift-modulation invariant spaces on LCA groups}

\author[C. Cabrelli]{Carlos Cabrelli}
\address{\textrm{(C. Cabrelli)}
Departamento de Matem\'atica,
Facultad de Ciencias Exac\-tas y Naturales,
Universidad de Buenos Aires, Ciudad Universitaria, Pabell\'on I,
1428 Buenos Aires, Argentina and
IMAS-CONICET, Consejo Nacional de Investigaciones
Cient\'ificas y T\'ecnicas, Argentina}
\email{cabrelli@dm.uba.ar}

\author[V. Paternostro]{Victoria Paternostro}
\address{\textrm{(V.Paternostro)}
Departamento de Matem\'atica,
Facultad de Ciencias Exac\-tas y Naturales,
Universidad de Buenos Aires, Ciudad Universitaria, Pabell\'on I,
1428 Buenos Aires, Argentina and
IMAS-CONICET, Consejo Nacional de Investigaciones
Cient\'ificas y T\'ecnicas, Argentina}
\email{vpater@dm.uba.ar}

 \thanks{The research of
the authors is partially supported by
Grants: CONICET, PIP 398, UBACyT X638 and X502
}

\date{\today}

\begin{abstract}
A $(K,\Lambda)$ shift-modulation invariant space  is a subspace of $L^2(G)$, that is invariant by translations along elements in $K$ and modulations by elements in $\Lambda$. Here $G$ is a locally compact
abelian group, and $K$ and $\Lambda$ are closed subgroups of $G$ and  the dual group $\hat G$, respectively.

In this article we provide a characterization of
shift-modulation invariant spaces in this general context when $K$ and $\Lambda$ are uniform lattices. This extends previous results known for $L^2(\R^d)$.
We develop fiberization techniques  and suitable range functions
adapted to LCA groups  needed  to provide the  desired characterization.
\end{abstract}



\maketitle

\section{Introduction}
Shift-invariant spaces (SIS)  play an important role in approximation theory, wavelets and frames.
They have also proven to be very useful as models in  signal processing applications.

Shift-modulation invariant (SMI) spaces are  shift invariant spaces that have the extra condition to be  also
 invariant  under some group of modulations. These shift invariant spaces with the extra assumption of modulation invariance are of particular interest and are usually known as
Gabor or Weyl-Heisenberg spaces. They have
been intensively studied in \cite{Bow07}, \cite{CC01b}, \cite{CC01a}, \cite{Chr03}, \cite{Dau92}, \cite{GD04},  \cite{GD01}, \cite{Gro01}.

A very deep and detailed  study  of the structure of shift-modulation invariant spaces of $L^2(\R^d)$,  was
 given by Bownik (see \cite{Bow07}).
In that work, the author provides a characterization of SMI spaces based on fiberization techniques and
 range functions.

In the case of SIS,  the $L^2(\R^d)$ theory using range functions has been extended to the context of
locally compact abelian (LCA) groups  in \cite {CP10} and  \cite{KR08}.
This  general framework allows for a more complete view, in which relationships  among the  groups involved and its properties are more clearly exposed. Further, the LCA group setting  includes
the finite case. Having a valid theory for groups such as $\Z_n$ is important for applications.

Since modulations become  translations in the Fourier domain, shift-modu\-la\-tions invariant spaces are spaces that are
 shift-invariant in time and frequency.
As a consequence the techniques of shift-invariant spaces can be applied to study the structure of SMI spaces.
Having at hand  a theory of SIS on LCA groups, it is natural  to ask whether a
general theory of SMI spaces could be developed in this more general context.

In this article we define and study the structure of SMI spaces using range functions and fiberization techniques, in the context of LCA groups.
First we introduce the notion of shift-modulation spaces
where translations are on a closed
subgroup of an LCA group $G$ and modulations are on a closed subgroup of the dual group  $\widehat {G}$.
Next we focus our attention to the case where both, translations and modulations, are along uniform lattices of $G$ and  the dual group of $G$ respectively, with some minor hypotheses.
We prove a characterization  of shift-modulation invariant spaces,  extending  the result obtained by Bownik in \cite{Bow07} for the case of  $L^2(\R^d)$ to the case of LCA groups.
 The LCA setting of this article  allows to visualize the role played by the different components: the groups, their duals,
the quotiens, and the relationships between the different lattices involved and their annihilators. This role, is somehow hidden in the Euclidean  case.
For example in the classic case, the simple structure of the  lattices  allows to reduce
the study to the case when one of the lattices is $\Z^d$. Since the dual group of $\R^d$ is isomorphic to itself
and the annihilator of $\Z^d$ is isomorphic to  $\Z^d$, most of the analysis  can be done in the original group $\R^d$. 
The fact that in the general case we do not have these isomorphisms, requires some effort to discover the role of each component.
The lack of structure of the lattices involved creates additional difficulties in establishing the precise  setting.
This is particularly relevant in defining the different isomorphisms that lead to a suitable  decomposition
of $L^2(G)$ and consequently to a range function well adapted to spaces having this double invariance.
A diversity  of results are then required such as the existence of Borel sections and their relations (see for example equations (\ref{sec-delta}) and (\ref{sec-H})).

On the other side, once  the proper setting is  obtained, there is a more clear picture of the general interrelationships,
that have the additional advantage  to simplify part of treatment.

We have organized the article as follows. In  Section \ref{sec-2} we review some of the basic facts on LCA groups. Then we develop the notion of shift-modulation invariant spaces on this context and we set our standing assumptions. In Section \ref{sec-heuristic} we outline how the results on shift invariant spaces can be used for shift-modulation invariant spaces.  Section \ref{sec-fibras-rango} establishes the fiberization isometry and the concept of shift-modulation range functions. Section \ref{section-5} 
contains the main result of the paper, that is, the characterization of shift-modulation invariant spaces under uniform lattices. Finally, we include some examples in Section \ref{sec-example}.

\section{Preliminaries}\label{sec-2}

In this section we will set the known results and notation needed for the paper. We will also state our standing assumptions that will be in force for the remainder of this manuscript.

Let $G$ be an arbitrary  locally compact Hausdorff abelian  group  written additively.
We will denote  by $m_G$ its Haar measure. The dual group of $G$, that is, the set of continuous characters on $G$,  is denoted  by $\G$ or $\widehat{G}$. The value of the character $\g\in\G$ at the point $x\in G$ is written by $(x,\g)$.

When two LCA groups $G_1$ and $G_2$ are topologically isomorphic we will write  $G_1\approx G_2$.
In particular, it is known that $\widehat{\G}\approx G$, where $\widehat{\G}$ is the dual of the dual group of $G$.

For an LCA group $G$ and  $K\subset G$ a closed subgroup of $G$
the Haar measures $m_G$, $m_K$ and $m_{G/K}$  can be chosen such that  Weil's Formula
\begin{equation}
\int_G f(x)\,dm_G(x)=\int_{G/K} \int_K f(x+k)\,dm_K(k)\,dm_{G/K}([x]),
\end{equation}\label{wiel-formula}
holds for each $f\in L^1(G)$. Here $[x]$ denotes the coset of $x$ in the quotient $G/K$.
Given a subgroup $K$ of an LCA group $G$ we will indicate by $\Pi_{G/K}$ a  section for the quotient $G/K.$

The Fourier transform of a Haar integrable function $f$ on $G$ is the function $\widehat{f}$ on $\G$ defined by
$$\widehat{f}(\g)=\int_G f(x)(x,-\g)\,dm_G(x),\quad\g\in \G.$$
When the Haar measures $m_G$ and $m_{\G}$ are normalized such that the inversion formula holds (see \cite{Rud62}),
the Fourier transform on $L^1(G)\cap L^2(G)$ can be extended
to a unitary operator from $L^2(G)$ onto $L^2(\G)$, the so-called
Plancharel transformation. We also denote this transformation by ``$\land$".

If $K$ is a subgroup of $G$, the subgroup of $\G$
$$K^*=\{\gamma \in \G: (k,\g)=1, \,\,\forall\,\, k\in K\}$$
is called the {\it annihilator} of $K$. Since every character in $\G$ is continuous, $K^*$ is a closed subgroup of $\G$. When $K$ is a closed subgroup, it is true that $(K^*)^*\approx K$.
For every closed subgroup $K$ of $G$ the following duality relationships are valid:
\begin{equation}\label{rel-dualidad}
K^*\approx\widehat{(G/K)}\qquad\textrm{ and }\qquad\G/K^*\approx \widehat{K}.
\end{equation}

\begin{definition}
Given $G$ an LCA group, a  {\it uniform lattice $K$ in $G$} is a discrete subgroup of $G$ such that the quotient group $G/K$
is  compact.
\end{definition}
It is known that, for a uniform lattice $K$ in $G$,  there exists a measurable (Borel) section
of $G/K$ with finite $m_{G}$-measure (see \cite{KK98}
and \cite{FG64}).
Another important fact about uniform lattices is that if $K$ is a countable (finite or countably infinite) uniform lattice in $G$, then
$K^*$ is a countable uniform lattice in $\G$.

\begin{remark}\label{lattices-finito}
If $K_1\subset K_2$ are lattices in $G$, then $K_2/K_1$ is finite. To prove this,
observe that, since $K_2^*\subset K_1^*$, $\widehat{K_1^*}/\widehat{K_2^*}\approx K_2/K_1$ due to the duality relationships stated in (\ref{rel-dualidad}). Therefore, $K_2/K_1$ is both compact and discrete. Hence $K_2/K_1$ must be finite.
\end{remark}

Let $G$ be an  LCA group, $K$ a countable uniform lattice  in $G$ and $\Pi_{\G/K^*}\subseteq \G$ a $m_{\G}$-measurable section of the quotient $\Gamma/K^*$.
Throughout this article we will identify the space  $L^p(\Pi_{\G/K^*})$
with the set $\{\p\in L^p(\G): \p=0\,\, a.e.\,\,\, \G\setminus\Pi_{\G/K^*}\}$ for $p=1$ and $p=2$.

\begin{proposition}
Let $G$ be an LCA group and $K$ a countable uniform lattice in $G$.
Then, $\{\eta_k\}_{k\in K}$ is an orthogonal basis for $L^2(\Pi_{\G/K^*})$, where
$\eta_k(\g)=(k,-\g)\chi_{\Pi_{\G/K^*}}(\g)$.

Moreover, when $m_{\G}(\Pi_{\G/K^*})=1$,
$\{\eta_k\}_{k\in K}$ is an orthonormal basis for $L^2(\Pi_{\G/K^*})$.
\label{lema-bon}
\end{proposition}

The following proposition will be needed later. Its proof can be found in \cite{CP10}.

\begin{proposition}
Let $G$ be an LCA group and  $K$ a countable uniform lattice
on  $G$.
Fix $\Pi_{\G/K^*}$ a Borel section of $\,\G/K^*$ and  choose
$m_{K^*}$ and $m_{\G/K^*}$ such that the inversion formula holds.
Then
$$\|a\|_{\ell^2(K)}=
\frac{m_K(\{0\})^{1/2}}{m_{\G}(\Pi_{\G/K^*})^{1/2}}\|\sum_{k\in K}a_k\eta_k\|_{L^2(\Pi_{\G/K^*})},$$
for each  $a=\{a_k\}_{k\in K}\in\ell^2(K)$.
\label{lema-parseval}
\end{proposition}

\begin{definition}\label{sm-def}
If $K\subseteq G$ and $\Lambda\subseteq \G$ are closed subgroups, 
we will say that a closed subspace  $V\subset L^2(G)$ is:
\begin{enumerate}
\item 
{\it $K$-shift invariant} (or
{\it shift invariant under $K$})
if
$$
f\in V \Rightarrow T_kf\in V\quad\forall\,\,k\in K,\,\quad \textrm{where} \quad T_kf(x)=f(x-k).
$$
\item
{\it $ \Lambda$-modulation invariant} (or
{\it modulation invariant under $\Lambda$})
if
$$
f\in V \Rightarrow M_{\la}f \in V\quad \forall \,\, \la \in \Lambda, \,\quad \textrm{where} \quad M_{\la}f(x)=(x,\la)f(x).
$$
\item
 {\it $(K, \Lambda)$--invariant} (or
{\it shift-modulation invariant under $(K, \Lambda)$})
if $V$ is shift invariant under $K$ and modulation invariant under $\Lambda$.
In that case
$$ f\in V \Rightarrow M_{\la}T_kf\in V\quad\forall\,\,k\in K\,\textrm{and}\,\la\in\Lambda.$$
\end{enumerate}
\end{definition}


For a subset $\mathcal{A}\subset L^2(G)$, define
$$E_{(K,\Lambda)}(\A)= \{M_{\la}T_k\p: \p\in\A, k\in K, \la\in\Lambda\}$$
and
$$S_{(K,\Lambda)}(\A)=\clspan \,E_{(K,\Lambda)}(\A).$$
A straightforward computation shows that the space  $S_{(K,\Lambda)}(\A)$ is shift-modulation  invariant under the pair $(K, \Lambda)$. We call $S_{(K,\Lambda)}(\A)$
the $(K, \Lambda)$-invariant space generated by $\A$.  Note that for every 
 $(K, \Lambda)$-invariant space $V$,
there exists  a countable set of generators $\mathcal{A}\subset L^2(G)$ such that
$V=S_{(K,\Lambda)}(\A)$. 

Similarly, $S_K(\A)$ will denote the closed subspace generated by 
translations  along $K$ of the elements of $\A$ and   $S_{\Lambda}(\A)$ the closed subspace generated by modulations from $\Lambda$ of the elements of $\A$.

In this paper we characterize  $(F,\Dd)$-invariant spaces for the case in which $F$ and $\Dd$  are uniform lattices
in $G$ and $\G$ respectively
and  $F\cap\Dd^*$ is a uniform lattice in $G$.
The condition on $F\cap\Dd^*$ of being a uniform lattice corresponds in the classic case $G=\R^d$ to the
rationally dependent lattices, see  \cite{Bow07}.
Similar to the shift invariant case (see \cite{CP10}), the  characterization of shift-modulation invariant spaces will be established in terms of appropriate range functions using fiberization techniques.

We  will set now our standing assumptions which will be in effect throughout the next sections.
\begin{standing}\label{standing}
\noindent
\begin{enumerate}
\item[$\bullet$] $G$ is a second countable LCA group and $\G$ its dual group.
\item[$\bullet$] $F$ is a countable uniform lattice on $G$. (translations)
\item[$\bullet$] $\Dd$ is a countable uniform lattice on $\G$. (modulations)
\item[$\bullet$] The sublattice $E:=F\cap\Dd^*$ is a (countable) uniform lattice on $G$.
\end{enumerate}
\end{standing}

As a consequence of our Standing Assumptions and Remark \ref{lattices-finito}
we obtain that:
\begin{enumerate}
\item [$(a)$]$E^*$ is a uniform lattice in $\G$ and $\Dd\subset E^*$.
\item [$(b)$]The quotient $E^*/\Dd$ is finite.
\item [$( c)$]$\Dd^*$ is a uniform lattice in $G$.
\end{enumerate}

Now we observe that if we  fix $\Pi_{\G/E^*}\subseteq \G$ a measurable
section for the quotient
$\G/E^*$ and
$\Pi_{E^*/\Dd^*}\subseteq E^*$ a finite section for $E^*/\Dd$,  then, we
can construct a measurable
section $\Pi_{\G/\Dd}$ for the quotient $\G/\Dd$ as
\begin{equation}
\Pi_{\Gamma/\Dd}=\bigcup_{e\,\in \,\Pi_{E^*/\Dd}}\Pi_{\G/E^*}+e.
\label{sec-delta}
\end{equation}
Let $\Pi_{F/E} \subseteq F$ be a finite section
for $F/E$. Note that, by the first isomorphism theorem for groups, $F/E$ is isomorphic to
$(F+\Dd^*)/\Dd^*$. Thus, $\Pi_{F/E} \subseteq F$ is also a section for $(F+\Dd^*)/\Dd^*$. Then, considering  $\Pi_{G/(F+\Dd^*)}$ a measurable section for $G/(F+\Dd^*)$, we
have that
\begin{equation}
\Pi_{G/\Dd^*}=\bigcup_{d\,\in\, \Pi_{F/E} }\Pi_{G/(F+\Dd^*)}-d
\label{sec-H}
\end{equation}
is a section for the quotient $G/\Dd^*$.
The minus sign in formula (\ref{sec-H}) is just for notational convenience in what follows.

These sections will be used to define the fiberization isometry and the range function.

In order to avoid carrying  over constants through the article, we will fix the following normalization of the Haar measures of the groups considered in this paper. This particular choice of the Haar measures does not affect the generality of our results.

First, we choose $m_{\Dd^*}$ such that $m_{\Dd^*}(\{0\})=1$. Then we fix $m_G$ and $m_{G/\Dd^*}$ such that the Weil's formula holds among $m_{\Dd^*}$,  $m_G$ and $m_{G/\Dd^*}$. Furthermore, we choose $m_{\G/E^*}$, $m_{E^*}$ in order to get
$m_{E^*}(\{0\})m_{\G/E^*}(\G/E^*)=\frac1{|\Pi_{E^*/\Dd^*}|}$ where by
$|\Pi_{E^*/\Dd^*}|$ we denote the cardinal of $\Pi_{E^*/\Dd^*}$. Then, we choose $m_{\G}$ such that  Weil's formula holds
among $m_{\G/E^*}$, $m_{E^*}$ and $m_{\G}$.

If $\Pi_{\G/\Dd}$ is given by (\ref{sec-delta}), this normalization implies
that $m_{\G}(\Pi_{\G/\Dd})=1$. This is  due to the formula $m_{\G}(\Pi_{\G/E^*})=m_{E^*}(\{0\})m_{\G/E^*}(\G/E^*)$ proved in \cite[Lemma 2.10]{CP10}.

In the paper we will use different instances of the following space.
\begin{definition}\label{hilbert-space}
Let $(X, \mu)$ be a finite measure space and $\mathcal{H}$ a separable Hilbert space with inner product $\langle\cdot,\cdot\rangle_{\mathcal{H}}$. We define $L^2(X,\mathcal{H})$ as the space of all measurable functions $\Phi:X\rightarrow\mathcal{H}$ such that
$$\|\Phi\|_2^2:=\int_{X}\|\Phi(x)\|_{\mathcal{H}}^2\,d\mu(x)<\infty,$$
where $\Phi:X\rightarrow\mathcal{H}$ is measurable if for each $v\in\mathcal{H}$ , the function
$x\mapsto\langle\Phi(x),v\rangle_{\mathcal{H}}$ from $X$ to $\C$ is measurable in the usual sense.
The space $L^2(X,\mathcal{H})$, with the inner product
$$\langle \Phi, \Psi\rangle:=\int_{X}\langle\Phi(x),\Psi(x)\rangle_{\mathcal{H}}\,d\mu(x)$$
is a complex Hilbert space.
\end{definition}

\section{Heuristic}\label{sec-heuristic}

Before we proceed to state the results and their proofs we will give in this section an informal discussions of the main ideas.
We start including  some results   from \cite{CP10} that we need.
We will only mention those that we require for our characterization.
For further details please refer to  \cite{CP10}.

Let $G$ be an LCA group, $H$ a uniform lattice in $G$ 
 and $\Pi_{\G/H^*}$ a measurable section for $\G/H^*$. Then, the spaces 
$L^2(G)$ and  $L^2(\Pi_{\Gamma/H^*},\ell^2({H^*})$ are isometrically isomorphic via the isomorphism:
\begin{equation}\label{iso2}
\T_{H}:L^2(G)\longrightarrow L^2(\Pi_{\Gamma/H^*},\ell^2({H^*})), \qquad \T_{H}f(\gamma)= \{\widehat{f}(\gamma+\delta)\}_{\delta \in H^*}
\end{equation}
The object $\{\widehat{f}(\g+\dd)\}_{\dd\in H^*}$ is called the $H^*$-{\it fiber} of $\widehat{f}$ 
at $\g$.

The isometry $ \T_H$ defined by (\ref{iso2}) is used to characterize by means of range functions, subspaces of  $L^2(G)$ that are shift invariant under $H$ :

\begin{definition}
A {\it shift range function} with respect to the pair $(G,H)$ is
a mapping
$$J:\Pi_{\G/H^*}\longrightarrow \{\textrm{closed subspaces of}\,\ell^2({H^*})\}.$$
The space $J(\gamma)$ is called the {\it fiber space}  associated to  $\gamma$.

\end{definition}

Then we have the characterization:
\begin{theorem}\cite[Theorem 3.10]{CP10}\label{SIS}
Let $V$ be a closed subspace of $L^2(G)$. Then $V$ is $H$-shift invariant if and only if there exists a measurable
shift range function $J$ such that
$$V=\big\{f\in L^2(G):\,\, \T_{H} f(\gamma)\in J(\gamma)\,\,\textrm{a.e.}\,\,\gamma\in \Pi_{\G/H^*}\}.$$
If $\A$ is a set of generators of $V$ as $H$-shift invariant  (i.e. $V=S_H(\A)$), then for a.e. $\g\in\Pi_{\G/H^*}$
$$J(\g)=\clspan\{\T_H\p(\g)\,:\, \p\in\A\}.$$ 
\end{theorem}

Basically the result establishes that a function  $f \in L^2(G)$  belongs to $V$ if and only if each
fiber of the Fourier transform of $f$ is in the corresponding fiber space.

From Theorem \ref{SIS}, we can easily derive a similar result for {\it modulation} invariant spaces.
That is, if $\Lambda$ is a uniform lattice in $\Gamma$, and $W \subset L^2(G)$ is  $\Lambda$-modulation invariant, then
$\widehat{W},$ the image of $W$ under the Fourier transform is  $\Lambda$-{\emph{shift}} invariant   in $L^2(\Gamma)$.
Hence, fixing $ \Pi_{G/\Lambda^*}$a section for $G/\Lambda^*$ and using Theorem \ref{SIS},  we can derive the following characterization of modulation invariant spaces:
\begin{corollary}\label{MIS}
Let $W$ be a closed subspace of $L^2(G)$. Then $W$ is $\Lambda$-modulation invariant if and only if there exists a measurable
shift range function $J$ with respect to the pair $(\Gamma,\Lambda)$ such that
$$W=\big\{f\in L^2(G):\,\,  \widetilde{\T}_{\Lambda^*} f(x)\in J(x)\,\,\textrm{a.e.}\,\,x\in \Pi_{G/\Lambda^*}\},$$
where $\widetilde{\T}_{\Lambda^*}$ is the isomeric isomorphism 
\begin{equation}\label{iso1}
\widetilde{\T}_{\Lambda^*}:L^2(G)\longrightarrow L^2(\Pi_{G/\Lambda^*},\ell^2(\Lambda^*)), \qquad \widetilde{\T}_{\Lambda^*} f(x)= \{f(x+h)\}_{h\in \Lambda^*}.
\end{equation}
If $\A$ is a set of generators of $W$ as $\Lambda$-modulation invariant space  (i.e. 
$V=S_\Lambda(\A)$), then for a.e. $x\in\Pi_{G/\Lambda^*}$
$$J(x)=\clspan\{\widetilde{\T}_{\Lambda^*}\p(x)\,:\, \p\in\A\}.$$ 
\end{corollary}
That is, a function $f$ belongs to $W$ if and only if its fibers (an not the fibers of its  Fourier transform) belong to corresponding fiber space.

 Now we will describe how we will apply these results to shift-modulation invariant spaces.

Let $G$ be an LCA group and $\Gamma$  its dual, and let $F$ and $\Delta$ be uniform lattices in  $G$ and $\Gamma$ respectively satisfying Standing Assumptions (\ref{standing}).

Let $V$ be a $(F,\Delta)$-shift-modulation invariant space in $L^2(G)$, (see Definition \ref{sm-def}).
We have now two ways to characterize $V$, one using the invariance under translations and the other using the invariance under modulations.
Assume we choose the characterization using that our space  is  $\Delta$-modulation invariant. Then,
by Corollary \ref{MIS}, we have
a range function defined in a section $\Pi_{G/\Delta^*}$, with fiber spaces $J(x) \subset \ell^2(\Delta^*)$.

Next we will show that the invariance under translations along the lattice  $E =F\cap\Dd^*$ implies
that the fiber spaces $J(x)$ are $E$-shift invariant in $\ell^2(\Delta^*)$. So using again Theorem \ref{SIS},
we obtain a range function for each of the fiber spaces $J(x)$.
From there we construct a {\it shift-modulation} range function that will produce the desired characterization.

Note that we did not consider in this description  translations by elements of $F$ that are not in $E$.
As we will see later, the action of these elements produces some periodicity on the range function.

\section{The Fiberization Isometry and  Range Functions}\label{sec-fibras-rango}

The goal of this section is to define the fiberization isometry  and a suitable range function required to achieve the characterization of $(F,\Dd)$-invariant spaces. We start defining  the  isometry that will produce a decomposition of the space $L^2(G)$, and we show its relation with the Zak transform in Section \ref{isometry-zak}. In Section \ref{sec-range-function} we will first introduce the concept of shift-modulation range function. Then we  associate a shift-modulation invariant space to this range function and also  construct a range function from a given shift-modulation invariant space.

\subsection{ The Isometry}\label{isometry-zak}

Let us fix $F\subseteq G$ and $\Dd\subseteq \G$ countable uniform lattices verifying the Standing Assumptions
(\ref{standing}).

In order to construct the fiberization isometry, we must introduce  the following isomorphisms.

Let $\widetilde{\T}_{\Dd^*}:L^2(G)\longrightarrow L^2(\Pi_{G/\Dd^*},\ell^2(\Dd^*))$ be the isometric isomorphism  defined as in (\ref{iso1}) for $G$ and $\Dd^*$. That is,
\begin{equation}
\widetilde{\T}_{\Dd^*} f(x)=\{f(x+h)\}_{h\in \Dd^*}.
\end{equation}\label{tau-H}

On the other hand, consider
$\T_E:\ell^2(\Dd^*)\longrightarrow L^2(\Pi_{\G/E^*},\ell^2(\Pi_{E^*/\Dd}))$ defined by
\begin{equation}\label{tau-E}
\T_Ea(\xi)=\{\sum_{h\in \Dd^*}a_h\eta_h(\xi+e)\}_{e\in \Pi_{E^*/\Dd}},
\end{equation}
where the functions $\eta_h$ are as in (\ref{lema-bon}) and $a=\{a_h\}_{h\in \Dd^*}$.

\begin{lemma}
The map $\T_E$ defined in (\ref{tau-E}) is an isometric isomorphism between  $\ell^2(\Dd^*)$ and  $L^2(\Pi_{\G/E^*},\ell^2(\Pi_{E^*/\Dd}))$.
\end{lemma}

\begin{proof}
Since $\Pi_{E^*/\Dd}$ is an index set, according to Definition (\ref{hilbert-space}) we have that
\begin{eqnarray*}
\|\T_Ea\|^2_2&=&\int_{\Pi_{\G/E^*}}\sum_{e\in \Pi_{E^*/\Dd}}|\sum_{h\in \Dd^*}a_h\eta_h(\xi+e)|^2\,dm_{\G}(\xi)\\
&=&\int_{\Pi_{\G/\Dd}}|\sum_{h\in \Dd^*}a_h\eta_h(\w)|^2\,dm_{\G}(\w)\\
&=&\|\sum_{h\in \Dd^*}a_h\eta_h\|^2_{L^2(\Pi_{\G/\Dd})}.
\end{eqnarray*}

Now, applying Proposition \ref{lema-parseval} we obtain
$$\|\sum_{h\in \Dd^*}a_h\eta_h\|^2_{L^2(\Pi_{\G/\Dd})}=\frac{m_{\G}(\Pi_{\G/\Dd})}{m_{\Dd^*}(\{0\})}\|a\|^2_{\ell^2(\Dd^*)}.$$

Hence, by our normalization of the Haar measures,   $\frac{m_{\G}(\Pi_{\G/\Dd})}{m_{\Dd^*}(\{0\})}=1$ and then
$\|\T_Ea\|^2_2=\|a\|^2_{\ell^2(\Dd^*)}.$

Let $\Phi\in L^2(\Pi_{\G/E^*},\ell^2(\Pi_{E^*/\Dd}))$. Then $\Phi$ induces the function $\widetilde{\Phi}\in L^2(\Pi_{\G/\Dd})$
given by
$$\widetilde{\Phi}(\w)=\big(\Phi(\xi)\big)_e,$$
where $\w=\xi+e\in\Pi_{\G/\Dd}$, with $\xi\in \Pi_{\G/E^*}$ and $e\in \Pi_{E^*/\Dd}$. Here $\big(\Phi(\xi)\big)_e$ denotes the value of the sequence $\Phi(\xi)$ at $e$.
It is easy to check that $\|\Phi\|_2=\|\widetilde{\Phi}\|_{L^2(\Pi_{\G/\Dd})}$.

By Proposition \ref{lema-bon}, $\{\eta_h\}_{h\in \Dd^*}$ is  an orthonormal basis for $L^2(\Pi_{\G/\Dd})$.
Thus, $\widetilde{\Phi}=\sum_{h\in \Dd^*}a_h\eta_h$ for some $a=\{a_h\}_{h\in \Dd^*}\in\ell^2(\Dd^*)$. From this, it follows
that $\T_E a=\Phi$.
Therefore, $\T_E$ is an isomorphism.
\end{proof}

\begin{remark}\label{remark}
Note that $E^{*_{(\Dd^*)}}$, the annihilator of $E$ as a subgroup of $\Dd^*$, is
topologically isomorphic to $E^*/\Dd$. Hence, using the dual relationship stated in (\ref{rel-dualidad}), it follows that $\widehat{\Dd^*}/E^{*_{(\Dd^*)}}\approx \G/E^*$.
This allows us to look at $\T_E$ as  a particular case of the map defined in (\ref{iso2}).
\end{remark}

The isometric isomorphism $\T_E$ induces another isometric isomorphism
$$\Psi_1: L^2(\Pi_{G/\Dd^*},\ell^2(\Dd^*))\longrightarrow L^2(\Pi_{G/\Dd^*},L^2(\Pi_{\G/E^*},\ell^2(\Pi_{E^*/\Dd})))$$
defined by
$$\Psi_1(\phi)(x)=\T_E(\phi(x)).$$

In addition, we can identify the Hilbert space $L^2(\Pi_{G/\Dd^*},L^2(\Pi_{\G/E^*},\ell^2(\Pi_{E^*/\Dd})))$ with
$L^2(\Pi_{G/\Dd^*}\times \Pi_{\G/E^*},\ell^2(\Pi_{E^*/\Dd})))$ using the isometric isomorphism
$$ \Psi_2: L^2(\Pi_{G/\Dd^*},L^2(\Pi_{\G/E^*},\ell^2(\Pi_{E^*/\Dd})))\longrightarrow L^2(\Pi_{G/\Dd^*}\times \Pi_{\G/E^*},\ell^2(\Pi_{E^*/\Dd}))
$$
given by
$$\Psi_2(\phi)(x,\xi)=\phi(x)(\xi).$$

\begin{definition}\label{def-tau}
We define
$\T: L^2(G)\longrightarrow L^2(\Pi_{G/\Dd^*}\times \Pi_{\G/E^*},\ell^2(\Pi_{E^*/\Dd}))$
as
$$\T=\Psi_2\circ\Psi_1\circ\widetilde{\T}_{\Dd^*}.$$
\end{definition}

This mapping $\T$, which is actually an isometric isomorphism, and which we call  {\it  the fiberization isometry}, can be explicitly defined as
\begin{equation}
\T f(x,\xi)=\T_E(\widetilde{\T}_{\Dd^*}f(x))(\xi)=\{\sum_{h\in \Dd^*}f(x-h)(h,\xi+e)\}_{e\in \Pi_{E^*/\Dd}}.
\end{equation}

\subsubsection{The isometry and the Zak transform}\noindent

As it is well known, a natural tool to study shift-modulation invariant spaces is the Zak transform. 
The Zak transform was first introduced in $\R$ by Gelfand \cite{Gel50}. Then, Weil \cite{wei64} extended this transform to general LCA groups, and independently Zak \cite{Zak67} used it in physical problems. In what follows we show how the Zak transform is present in our analysis. 

We recall  the  usual Zak transform $Z:L^2(G)\to \mathcal{Q}$ given by
$$Zf(x,\xi)=\sum_{h\in \Dd^*}f(x-h)(h,\xi),$$
where $\mathcal{Q}$ is the set of all measurable functions $F:G\times \G\to \C$ satisfying
\begin{enumerate}
\item [$(a)$]$F(x+h,\xi)=(h,\xi)F(x,\xi) \,\,\forall\, h\in \Dd^*$,
\item [$(b)$]$F(x,\xi+\dd)=F(x,\xi)\,\, \forall\,\dd\in\Dd$  and
\item [$(c)$]$\|F\|^2=\int_{\Pi_{\G/\Dd}}\int_{\Pi_{G/\Dd^*}}|F(x,\xi)|^2\,dm_G(x)\,dm_{\G}(\xi)<\infty$.
\end{enumerate}

Then, it is clear that
$$\T f(x,\xi)=\{Zf(x,\xi+e)\}_{e\in \Pi_{E^*/\Dd}}.$$

The next lemma states an important property about $\T$ which  will be useful in what follows.
Its proof is a straightforward consequence of properties $(a)$, $(b)$ and $(c)$ formulated above.
\begin{lemma}\label{tau-propiedad}
For each $f\in L^2(G)$ the map $\T$ of Definition \ref{def-tau} satisfies
$$\T(M_{\dd}T_y f)(x,\xi)=(x,\dd)(-z,\xi)\T(T_d f)(x,\xi)\quad\textrm{ a.e. }(x,\xi)\in \Pi_{G/\Dd^*}\times \Pi_{\G/E^*} ,$$
where $\dd\in\Dd$, $y\in F$ and $y=z+d$ with $z\in E$ and $d\in \Pi_{F/E}$.
\end{lemma}

\subsection{Shift-modulation Range Functions}\label{sec-range-function}

In this section we will introduce the notion of  shift-modulation range function adapted to the defined isometry.
\begin{definition}
A {\it shift-modulation range function} with respect to the pair $(F,\Dd)$ is
a mapping
$$J:\Pi_{G/\Dd^*}\times\Pi_{\G/E^*}\longrightarrow \{\textrm{{\ subspaces of }}\,\ell^2(\Pi_{E^*/\Dd})\},$$
satisfying the following periodicity property:
\begin{equation}\label{J-periodica}
J(x,\xi)=J(x-d,\xi)\quad\forall\,\, d\in \Pi_{F/E}\;
\text{ and a.e. } (x, \xi)\in \Pi_{G/(F+\Dd^*)}\times\Pi_{\G/E^*}.
\end{equation}

\end{definition}
For a shift-modulation range function $J$, we associate  to each $(x,\xi)\in \Pi_{G/\Dd^*}\times\Pi_{\G/E^*}$
the orthogonal projection onto
$J(x,\xi)$, $P_{(x,\xi)}:  \ell^2(\Pi_{E^*/\Dd})\to J(x,\xi)$.

We say that a shift-modulation range function $J$ is measurable if the function
$(x,\xi)\mapsto P_{(x,\xi)}$ from $\Pi_{G/\Dd^*}\times\Pi_{\G/E^*}$ to $\ell^2(\Pi_{E^*/\Dd})$ is measurable.

For a shift-modulation range function $J$ (not necessarily measurable) we define the subset $M_J$ as
\begin{align}\label{m-j}
M_J=&\{\Psi\in  L^2(\Pi_{G/\Dd^*}\times \Pi_{\G/E^*},\ell^2(\Pi_{E^*/\Dd}))\\
&:\Psi(x,\xi)\in J(x,\xi), \textrm{a.e.}(x,\xi)\in \Pi_{G/\Dd^*}\times\Pi_{\G/E^*}\}.\nonumber\\
\nonumber
\end{align}

\begin{remark}\label{m-j-cerrado}
The subspace $M_J$ defined above is a closed subspace in $L^2(\Pi_{G/\Dd^*}\times \Pi_{\G/E^*},\ell^2(\Pi_{E^*/\Dd}))$. For the proof of this fact see \cite[Lemma 3.8]{CP10}.
\end{remark}

\subsubsection{The shift-modulation invariant space associated to a range function}\noindent

The following proposition states that if $J$ is a  given shift-modulation range function with respect
to the pair $(F,\Dd)$, we can associate to $J$
an $(F,\Dd)$-invariant space.

\begin{proposition}\label{prop-M-J-SMIS}
Let $J$ be a shift-modulation range function and define $V:=\T^{-1}M_J$, where $M_J$ is  as in
 (\ref{m-j}) and $\T$ is the fiberization isometry.

Then,
$V$ is an $(F,\Dd)$-invariant space in $L^2(G)$.
\end{proposition}

\begin{proof}
To begin with, observe that,  since $\T$ is an isometry, $V\subset L^2(G)$
is a closed subspace, by Remark \ref{m-j-cerrado}.

Let $f\in V$, $\dd\in\Dd$ and $y\in F$. We need to show that $M_{\dd}T_y f\in V$.

According to Lemma \ref{tau-propiedad}, we have that
$$\T(M_{\dd}T_y f)(x,\xi)=(x,\dd)(-z,\xi)\T(T_d f)(x,\xi)\quad\textrm{ a.e. }(x,\xi)\in \Pi_{G/\Dd^*}\times \Pi_{\G/E^*},$$
where $y=z+d$ with $z\in E$ and $d\in \Pi_{F/E}$.

In particular, if $x\in \Pi_{G/(F+\Dd^*)}$ we can rewrite $\T(T_d f)(x,\xi)$ as $\T f(x-d,\xi)$. Then, since
$\T f\in M_J$ and $J$ satisfies (\ref{J-periodica}),we have
$$\T(T_d f)(x,\xi)=\T f(x-d,\xi)\in J(x-d,\xi)=J(x,\xi),$$
for a.e. $(x,\xi)\in \Pi_{G/(F+\Dd^*)}\times \Pi_{\G/E^*}$.
Thus,
\begin{equation}\label{ecu-vale-en-F+H}
\T(M_{\dd}T_y f)(x,\xi)\in J(x,\xi)\quad\textrm{ a.e. }(x,\xi)\in \Pi_{G/(F+\Dd^*)}\times \Pi_{\G/E^*},
\end{equation}
and this is valid for all $y\in F$ and $\dd\in \Dd$.

We now want to show that (\ref{ecu-vale-en-F+H}) holds on $\Pi_{G/\Dd^*}\times \Pi_{\G/E^*}$.
Let $(x,\xi)\in \Pi_{G/\Dd^*}\times \Pi_{\G/E^*}$. By (\ref{sec-H}) we can set $x=x'-d$ with $x'\in \Pi_{G/(F+\Dd^*)}$
and $d\in \Pi_{F/E}$.
If we fix $\dd\in\Dd$ and $y\in F$, then
$\T(M_{\dd}T_y f)(x,\xi)=\T(T_dM_{\dd}T_y f)(x',\xi)$.
Since $M_{\la}T_k g=(k,\la)T_k M_{\la} g$ for all $g\in L^2(G)$, $\la\in \Dd$ and $k\in F$, we have
$$\T(T_dM_{\dd}T_y f)(x',\xi)=(-d,\dd)\T(M_{\dd}T_{d+y} f)(x',\xi)\in J(x',\xi)=J(x,\xi).$$
Then, (\ref{ecu-vale-en-F+H}) holds on $\Pi_{G/\Dd^*}\times \Pi_{\G/E^*}$. Therefore, $M_{\dd}T_y f\in V$ for all
$\dd\in\Dd$ and $y\in F$.
\end{proof}

\subsubsection{The range function associated to a $(F, \Dd)$-invariant space}
\noindent\label{construccion-J}

Theorem \ref{SIS} gives   a specific way to describe the shift range function associated to each shift invariant space in terms of the fibers of its  generators. The analogous description is given in Corollary \ref{MIS} for modulation invariant spaces. Now  we will use these results  to construct a shift-modulation range function from a given $(F, \Dd)$-invariant space. 

Assume that $V \subseteq L^2(G) $ is an $(F,\Dd)$-invariant space and that  $V=S_{(F,\Dd)}(\A)$
for some  countable set $\A\subseteq L^2(G)$. We will show, how to associate  to $V$
a shift-modulation range function.

Since $V$ is $\Dd$-modulation invariant, by Corollary \ref{MIS}, $V$ can be described as
\begin{equation}
V=\big\{f\in L^2(G):\,\, \widetilde{\T}_{\Dd^*} f(x)\in J_{\Dd^*}(x)\,\,\textrm{a.e.}\,\,x\in \Pi_{G/\Dd^*}\big\},
\label{V-H-inv}
\end{equation}
where $\widetilde{\T}_{\Dd^*}$ is the isometry defined in (\ref{tau-H}) and $J_{\Dd^*} $ is the  shift range function associated to $V$ given by
$$J_{\Dd^*}:\Pi_{G/\Dd^*} \longrightarrow \{\textrm{{closed subspaces of} } \ell^2(\Dd^*)\}$$
$$J_{\Dd^*}(x)=\sn\{\widetilde{\T}_{\Dd^*}(T_y\p)(x)\,:\,y\in F, \,\p\in\A\}.$$
Now, let us see that $J_{\Dd^*}(x)\subset \ell^2(\Dd^*)$ is a shift invariant space under translations in $E$.
Since  $\Pi_{F/E}\subseteq F$ is a section for the quotient $F/E$, every $y\in F$ can be written
in a unique way as $y=z+d$ with $z\in E$ and $d\in \Pi_{F/E}$. Then,   using that
$\widetilde{\T}_{\Dd^*} T_zf=T_z\widetilde{\T}_{\Dd^*}f$
for all $z\in E$, we can rewrite $J_{\Dd^*}(x)$ as
$$J_{\Dd^*}(x)=\sn\{T_z\widetilde{\T}_{\Dd^*}(T_d\p)(x)\,:\,z\in E,\, d\in \Pi_{F/E}, \,\p\in\A\}.$$
This description  shows that
 $J_{\Dd^*}(x)$ is  a shift invariant space under translations in $E$ generated
by the set $\{\widetilde{\T}_{\Dd^*}(T_d\p)(x)\,:\,d\in \Pi_{F/E}, \,\p\in\A\}$.

Using  Theorem \ref{SIS}, we can characterize $J_{\Dd^*}(x)$ for a.e $x\in \Pi_{G/\Dd^*}$ as follows.
For each $x\in \Pi_{G/\Dd^*}\setminus Z$, where $Z$ is the exceptional zero $m_G$-measure set,  there exists a range function $J_E^x:\Pi_{\G/E^*}\longrightarrow \{\textrm{{ subspaces of }}\,\ell^2(\Pi_{E^*/\Dd})\}$ such that
$$J_{\Dd^*}(x)=\big\{a\in \ell^2(\Dd^*):\,\, \T_E a(\xi)\in J_E^x(\xi)\,\,\textrm{a.e.}\,\,\xi\in \Pi_{\G/E^*}\big\},$$
where $\T_E$ is the map given in (\ref{tau-E}).

Moreover,
\begin{eqnarray*}
J_E^x(\xi)&=&\sn\{\T_E(\widetilde{\T}_{\Dd^*}T_d\p(x))(\xi)\,:\,d\in \Pi_{F/E}, \,\p\in\A\}\\
&=&\sn\{\T (T_d\p)(x, \xi)\,:\,d\in \Pi_{F/E}, \,\p\in\A\}\\
&=&\Span\{\T (T_d\p)(x, \xi)\,:\,d\in \Pi_{F/E}, \,\p\in\A\},
\end{eqnarray*}
where in the last equality we use that $\dim (\ell^2(\Pi_{E^*/\Dd}))<\infty$.

This leads to the function $J:\Pi_{G/\Dd^*}\times\Pi_{\G/E^*}\to \{\textrm{{subspaces of }}\,\ell^2(\Pi_{E^*/\Dd})\}$ defined as
\begin{equation}\label{range-function}
J(x,\xi)=\Span\{\T (T_d\p)(x, \xi)\,:\,d\in \Pi_{F/E}, \,\p\in\A\},
\end{equation}
for a.e. $(x,\xi)\in \Pi_{G/\Dd^*}\times\Pi_{\G/E^*}$.

\begin{lemma}\label{pinta-J}
Let $\A\subseteq L^2(G)$ a countable set.
Then, the map defined in (\ref{range-function}) is a shift-modulation range function.
\end{lemma}

\begin{proof}
We need to show that $J$ satisfies property (\ref{J-periodica}).

Let $d_0\in \Pi_{F/E}$.
For each $d\in \Pi_{F/E}$,  we have that  $\T(T_d\p)(x-d_0,\xi)=\T(T_{d+d_0}\p)(x,\xi)$ for a.e.
$(x,\xi)\in \Pi_{G/(F+\Dd^*)}\times \Pi_{\G/E^*}$.
Since $d+d_0\in F$,  it can be written as $d+d_0=d'+z'$ with $d'\in \Pi_{F/E}$ and
$z'\in E$. Then,  according to Lemma \ref{tau-propiedad},
$\T(T_{d+d_0}\p)(x,\xi)=(z',\xi)\T(T_{d'}\p)(x,\xi)$. Thus,
$\T(T_d\p)(x-d_0,\xi)\in J(x,\xi)$  due to $\T(T_{d'}\p)(x,\xi)\in J(x,\xi)$.
This shows that $J(x-d_0,\xi)\subseteq J(x,\xi)$ for a.e.
$(x,\xi)\in \Pi_{G/(F+\Dd^*)}\times \Pi_{\G/E^*}$ for each $d_0\in \Pi_{F/E}$.

With an analogous argument, it can be proven that
$J(x,\xi)\subseteq J(x-d_0,\xi)$ for a.e.
$(x,\xi)\in \Pi_{G/(F+\Dd^*)}\times \Pi_{\G/E^*}$ for each $d_0\in \Pi_{F/E}$.
\end{proof}

As we have seen in Proposition \ref{prop-M-J-SMIS}, each shift-modulation range function with respect to the pair $(F, \Dd)$ induces
an $(F,\Dd)$-invariant space. Furthermore, in Section \ref{construccion-J} we
associated to each shift-modulation invariant space $V$ a shift-modulation range function
from a  system of generators of $V$.
This leads to a natural question. If $V$ is an $(F,\Dd)$-invariant space and
$J$ the shift-modulation range function that $V$ induces, what is the relationship between $V$ and the
$(F,\Dd)$-invariant space induced from $J$?

That will be the content of the following section.
\section{The characterization of $(F,\Dd)$-Invariant Spaces}\label{section-5}

We can now state our main result which characterizes  $(F,\Dd)$-invariant spaces in terms of the fiberization isometry and shift-modulation range functions.

\begin{theorem}\label{thm-ppal}
Let $V\subset L^2(G)$ be a closed subspace and $\T$ the fiberization isometry of Definition \ref{def-tau}.
Then,
$V$ is an $(F,\Dd)$-invariant  space if and only if there exists a measurable shift-modulation range function
$J:\Pi_{G/\Dd^*}\times\Pi_{\G/E^*}\longrightarrow \{\textrm{{subspaces of }}\,\ell^2(\Pi_{E^*/\Dd})\}$ such that
$$V=\big\{f\in L^2(G):\,\, \T f(x,\xi)\in J(x,\xi)\,\,\textrm{a.e.}\,\,(x,\xi)\in \Pi_{G/\Dd^*}\times \Pi_{\G/E^*}\big\}.$$

Identifying shift-modulation  range functions which are equal almost everywhere, the correspondence between
 $(F,\Dd)$-invariant spaces and  measurable shift-modulation range functions is one to one and onto.

Moreover, if $V=S_{(F,\Dd)}(\A)\subseteq L^2(G)$ for some countable subset $\A$ of $L^2(G)$, the measurable shift-modulation range function $J$
associated to $V$ is given by
$$J(x,\xi)=\Span\{\T T_d\p(x, \xi)\,:\,d\in \Pi_{F/E}, \,\p\in\A\},$$
a.e. $(x, \xi)\in \Pi_{G/\Dd^*}\times\Pi_{\G/E^*}$.
\end{theorem}

For the proof of Theorem \ref{thm-ppal}
we need the following   lemma, which is an  adaptation  of
\cite[Lemma 3.11]{CP10}.
For its  proof see \cite{CP10}.

\begin{lemma}\label{range-iguales-ae}
If $J$ and $J'$ are two measurable shift-modulation range functions such that $M_J=M_{J'}$,
where $M_J$ and $M_{J'}$ are given by (\ref{m-j}),
then $J(x,\xi)=J'(x,\xi)$ a.e. $(x,\xi)\in \Pi_{G/\Dd^*}\times \Pi_{\G/E^*}$.
That is, $J$ and $J'$ are equal almost everywhere.
\end{lemma}

\begin{proof}[Proof of Theorem 4.1]
If  $V$ is an $(F,\Dd)$-invariant space, then, since $L^2(G)$ is separable, we have that
$ V=S_{(F,\Dd)}(\A)$ for some countable subset $\A$ of $L^2(G)$.

Let us consider the function $J$
defined as
$$J(x,\xi)=\Span\{\T (T_d\p)(x, \xi)\,:\,d\in \Pi_{F/E}, \,\p\in\A\}$$
 defined on $ \Pi_{G/\Dd^*}\times\Pi_{\G/E^*}$ and taking values in $\{\textrm{{ subspaces of }}\,\ell^2(\Pi_{E^*/\Dd})\}$.

 As a consequence of Lemma \ref{pinta-J}, $J$ is a shift-modulation range function. We must prove that
$\T V=M_J$ where $M_J$ is as in (\ref{m-j}) and that $J$ is measurable.

We will first show  $\T V=M_J$.

Take $\dd\in\Dd$, $y\in F$ written as $y=z+d$ with $z\in E$ and $d\in \Pi_{F/E}$, and $\p\in \A$. Then, by Lemma \ref{tau-propiedad} we have that
$$\T(M_{\dd}T_y \p)(x,\xi)=(x,\dd)(-z,\xi)\T(T_d \p)(x,\xi)\quad\textrm{ a.e. }(x,\xi)\in \Pi_{G/\Dd^*}\times \Pi_{\G/E^*}.$$
Thus, since $\T (T_d\p)(x, \xi) \in J(x,\xi)$, we have that $\T(M_{\dd}T_y \p)(x,\xi)\in J(x,\xi)$
a.e. $(x,\xi)\in \Pi_{G/\Dd^*}\times \Pi_{\G/E^*}$. Therefore,
$$\T(\text{span}\{M_{\dd}T_y\p: \p\in\A, y\in F, \dd\in\Dd\})\subseteq M_J.$$

Using that $\T$ is a continuous function and Remark \ref{m-j-cerrado}, we can compute
\begin{eqnarray*}
\T V&=&\T (\sn\{M_{\dd}T_y\p: \p\in\A, y\in F, \dd\in\Dd\})\\
&\subseteq&\overline{\T(\text{span}\{M_{\dd}T_y\p: \p\in\A, y\in F, \dd\in\Dd\})}\\
&\subseteq& \overline{M_J}=M_J.
\end{eqnarray*}

Let us suppose that $\T V\subsetneq M_J$. Then, there exists $\Psi\in M_J\setminus \{0\}$
orthogonal to $\T V$.
In particular, we have that $\langle \Psi, \T (M_{\dd}T_y\p)\rangle=0$ for all $\p\in\A, y\in F$ and
$\dd\in\Dd$.
Hence, if we write $y=z+d$ with $z\in E$ and $d\in \Pi_{F/E}$, by Lemma \ref{tau-propiedad} we obtain
\begin{eqnarray*}
0&=&\int_{\Pi_{G/\Dd^*}}\int_{\Pi_{\G/\Dd}}\langle \Psi(x,\xi), \T (M_{\dd}T_y\p)(x,\xi)\rangle\, dm_{\G}(\xi)\,dm_G(x)\\
&=&\int_{\Pi_{G/\Dd^*}}\int_{\Pi_{\G/\Dd}}(x,\dd)(-z,\xi)\langle \Psi(x,\xi), \T (T_d\p)(x,\xi)\rangle\, dm_{\G}(\xi)\,dm_G(x)\\
&=&\int_{\Pi_{G/\Dd^*}}\int_{\Pi_{\G/\Dd}}\eta_{\dd}(x)\eta_{-z}(\xi)\langle \Psi(x,\xi), \T (T_d\p)(x,\xi)\rangle\, dm_{\G}(\xi)\,dm_G(x),\\
\end{eqnarray*}
where $\eta_{\dd}$ and $\eta_{-z}$ are as in Proposition \ref{lema-bon}.

If we define $\nu_{(\dd,z)}(x,\xi):=\eta_{\dd}(x)\eta_{-z}(\xi)$, then, using Proposition \ref{lema-bon},
it can be seen that
$\{\nu_{(\dd,z)}\}_{(\dd,z)\in \Dd\times E}$ is an orthogonal basis for $L^2(\Pi_{G/\Dd^*}\times \Pi_{\G/E^*})$.
Therefore, $\langle \Psi(x,\xi), \T (T_d\p)(x,\xi)\rangle=0$ a.e. $(x,\xi)\in \Pi_{G/\Dd^*}\times \Pi_{\G/E^*}$ for all
$d\in \Pi_{F/E}$.

This shows that $\Psi(x,\xi)\in J(x,\xi)^{\perp}$
 a.e. $(x,\xi)\in \Pi_{G/\Dd^*}\times \Pi_{\G/E^*}$ and, since $\Psi\in M_J$ we must have $\Psi=0$,
 which is a contradiction. Thus $\T V=M_J$.

Let us prove now that $J$ is  measurable. If $\mathcal{P}$ is the orthogonal projection on $M_J$, $\mathcal{I}$ is the identity mapping in  $L^2(\Pi_{G/\Dd^*}\times \Pi_{\G/E^*},\ell^2(\Pi_{E^*/\Dd}))$ and
$\Psi\in L^2(\Pi_{G/\Dd^*}\times \Pi_{\G/E^*},\ell^2(\Pi_{E^*/\Dd}))$ we have that $(\mathcal{P}-\mathcal{I})\Psi$ is orthogonal to $M_J$. Then, with the above  reasoning
$(\mathcal{P}-\mathcal{I})\Psi(x,\xi)\in J(x,\xi)^{\perp}$ for a.e. $(x,\xi)\in \Pi_{G/\Dd^*}\times \Pi_{\G/E^*}$.
Thus,
$$P_{(x,\xi)}\big((\mathcal{P}-\mathcal{I})\Psi(x,\xi)\big)=0\quad\textrm{a.e.}\, (x,\xi)\in \Pi_{G/\Dd^*}\times \Pi_{\G/E^*}$$
and therefore, $\mathcal{P}\Psi(x,\xi)=P_{(x,\xi)}(\Psi(x,\xi))$ for a.e. $(x,\xi)\in \Pi_{G/\Dd^*}\times \Pi_{\G/E^*}$.
If in particular $\Psi(x,\xi)=a$ for all $(x,\xi)\in \Pi_{G/\Dd^*}\times \Pi_{\G/E^*}$, we have that
$\mathcal{P}a(x,\xi)=P_{(x,\xi)}(a)$. Therefore, since $(x,\xi)\mapsto\mathcal{P}a(x,\xi)$ is measurable,
$(x,\xi)\mapsto P_{(x,\xi)}a$ is measurable as well.

Conversely. If $J$ is a shift-modulation range function, by Proposition \ref{prop-M-J-SMIS}, $V:=\T^{-1}M_J$ is an $(F,\Dd)$-invariant space. Then,
$V=S_{(F,\Dd)}(\A)$ for some countable subset $\A$ of $L^2(G)$ and, by Lemma \ref{pinta-J} we can define the shift-modulation range function $J'$ as
$$J'(x,\xi)=\Span\{\T (T_d\p)(x, \xi)\,:\,d\in \Pi_{F/E}, \,\p\in\A\}\quad\textrm{a.e.}\, (x,\xi)\in \Pi_{G/\Dd^*}\times \Pi_{\G/E^*}.$$
Thus, as we have shown, $J'$ is measurable and $M_{J'}=\T V=M_J$. Then,  Lemma \ref{range-iguales-ae} gives us  $J=J'$ a.e.

This also proves that the correspondence between  $(F,\Dd)$-invariant spaces and  shift-modulation measurable range functions is one to one and onto.

\end{proof}

\begin{remark}
All the results of this paper are valid for uniform lattices satisfying Standing Assumptions \ref{standing}.
However, for shift-modulation invariant spaces where the translations (modulations) are along the whole group (dual group) we still can give a characterization as a corollary of Wiener's theorem
(see \cite{Hel64}, \cite{Sri64}, \cite{Rud87} and \cite{HS64}). 
\end{remark}

We say that a set $B\subset G$ is $K$-translation invariant if $B+k=B$ for all $k\in K$.

\begin{proposition}\label{prop-translation-total}
Let $V\subset L^2(G)$ be a closed subspace and $\Lambda\subset\G$ be a closed subgroup. Then,
$V$ is $(G,\Lambda)$-invariant if and only if there exists a $m_{\G}$-measurable set $B\subset\G$ which is $\Lambda$-translation invariant such that
$$V=\{f\in L^2(G)\,:\, \supp(\widehat{f})\subseteq B\}.$$
\end{proposition}
\begin{proof}
By Wiener's theorem, there exists a
$m_{\G}$-measurable set $B\subset \G$ satisfying
$V=\{f\in L^2(G)\,:\, \supp(\widehat{f})\subseteq B\}.$
Since $V$ is $\Lambda$-modulation invariant it follows that $B$ is $\Lambda$-translation invariant. 
\end{proof}

The following proposition is  analogous to the previous one for the case when the subspace is invariant along every modulation.
 
\begin{proposition}
Let $V\subset L^2(G)$ and $T\subset G$ be a closed subgroup. Then,
$V$ is $(T,\G)$-invariant  if and only if there exists a $m_{G}$-measurable set $A\subset G$ which is
$T$-translation invariant such that
$$V=\{f\in L^2(G)\,:\, \supp(f)\subseteq A\}.$$
\end{proposition}

Finally, we have the following corollary. 
\begin{corollary}
Let $V\subset L^2(G)$ be a non zero closed subspace. If
$V$ is $(G,\G)$-invariant, then $V=L^2(G)$.
\end{corollary}


\section{Examples}\label{sec-example}

In order to illustrate the constructions of the previous sections we now present some examples. 

\begin{example}\label{example-1}
Let $G=\R$. Then $\G=\R$. Now fix  $F=\frac{2}{3}\Z$ as the lattice for translations and $\Dd=\Z$
as the lattice for modulations. Since $\Dd^*=\Z$, then 
the lattice $E=F\cap\Dd^*$ is  $2\Z$ and  $F+\Dd^*=\frac{2}{3}\Z+\Z=\frac1{3}\Z$.
Thus, $E^*=\frac1{2}\Z$ and $E^*/\Dd\approx\Z_2$.
Hence, we can fix $\Pi_{\G/E^*}=[0,\frac1{2})$ and $\Pi_{E^*/\Dd}=\{0, \frac1{2}\}$. 

On the other hand, we set $\Pi_{F/E}=\{0, -\frac{2}{3}, -\frac{4}{3}\}$ and 
$\Pi_{G/(F+\Dd^*)}=[0, \frac1{3})$.
Then, by equation (\ref{sec-H}), 
$\Pi_{G/\Dd^*}= [0, \frac1{3})\cup [\frac{2}{3}, 1)\cup [\frac{4}{3}, \frac{5}{3})$ is a section for $G/\Dd^*$.

Therefore, the fundamental isometry of Definition \ref{def-tau} applied to $f\in L^2(\R)$ is given by the formula 
$$\T f(x,\xi)=(Zf(x,\xi), Zf(x, \xi+\frac1{2})),$$
where $x\in [0, \frac1{3})\cup [\frac{2}{3}, 1)\cup [\frac{4}{3}, \frac{5}{3})$, $\xi\in [0,\frac1{2})$ and 
$Z$ is the Zak transform in $\R$ given by the formula 
$Zf(x,\xi)=\sum_{k\in\Z}f(x-k)e^{2\pi ik\xi}$.

Let $\p\in L^2(\R)$ and let $S_{F,\Dd}(\p)$ be the $(F,\Dd)$-invariant space generated by $\{\p\}$. If $J$ is the shift-modulation range function associated to $S_{F,\Dd}(\p)$ through Theorem \ref{thm-ppal}, then
$$J(x,\xi)=\Span\{ \T\p(x, \xi),  \T\p(x+\frac{2}{3}, \xi),  \T\p(x +\frac{4}{3}, \xi)\},$$
where $x\in [0, \frac1{3})$ and $\xi\in [0,\frac1{2})$.
\begin{flushright}
$\square$
\end{flushright}

\end{example}

In the next example we  change the lattice of translations.  We will see that the fiberization isometry in Example \ref{example-2} has the same formula as in Example \ref{example-1}. 

\begin{example}\label{example-2}
Consider now $G=\R$, $F=\frac{2}{5}\Z$ and $\Dd=\Z$. Then, $E=2\Z$. With the same reasoning as in Example \ref{example-1} we have that
$\Pi_{G/\Dd^*}=[0,\frac1{5})\cup[\frac{2}{5}, \frac{3}{5})\cup[\frac{4}{5}, 1)\cup[\frac{6}{5},\frac{7}{5})
\cup[\frac{8}{5},\frac{9}{5})$ and $\Pi_{\G/E^*}=[0, \frac1{2})$.
Then, if $\p\in L^2(\R)$ the fiberzation isometry is 
$$\T \p(x,\xi)=(Z\p(x,\xi), Z\p(x, \xi+\frac1{2})),$$
where $x\in [0,\frac1{5})\cup[\frac{2}{5}, \frac{3}{5})\cup[\frac{4}{5}, 1)\cup[\frac{6}{5},\frac{7}{5})
\cup[\frac{8}{5},\frac{9}{5})$, $\xi\in [0,\frac1{2})$ and $Z$ is the usual Zak transform in $\R$.

In this case, since $\Pi_{F/E}=\{0, -\frac{2}{5}, -\frac{4}{5}, -\frac{6}{5}, -\frac{8}{5}\}$, the shift-modulation range function associated to $S_{F,\Dd}(\p)$ is 
$$J(x,\xi)=\Span\{ \T\p(x, \xi),  \T\p(x+\frac{2}{5}, \xi),  \T\p(x +\frac{4}{5}, \xi),  
\T\p(x +\frac{6}{5}, \xi),  \T\p(x +\frac{8}{5}, \xi)\},$$
where $x\in [0, \frac1{5})$ and $\xi\in [0,\frac1{2})$.
\begin{flushright}
$\square$
\end{flushright}

\end{example}

\begin{remark}
Note that in Examples \ref{example-1} and \ref{example-2}, the fiberization isometries have the same formula but different domains. This is due to the fact that the lattice $E$ is the same in both cases. The difference here appears in the translations that are outside of $E$ and they are mainly reflected in the shift-modulation range function.  
\end{remark}

Our last example is for $G=\Tor$.

\begin{example}
Let $G=\Tor$. Then, $\G=\Z$. Let us fix $m, n\in\N$. We consider $F=\frac1{m}\Z_m=\{0, \frac1{m}, \frac{2}{m}, \dots, 
\frac{m-1}{m}\}\subseteq \Tor$ as the lattice for translations and $\Dd=n\Z$ as the lattice for modulations. 
Since $\Dd^*=\frac1{n}\Z_n=\{0, \frac1{n}, \frac{2}{n}, \dots, 
\frac{n-1}{n}\}$, we have that $E=F\cap\Dd^*=\frac1{(n:m)}\Z_{(n:m)}$ where $(n:m)$ is the greatest common divisor between $n$ and $m$, and that $F+\Dd^*=\frac1{[n:m]}\Z_{[n:m]}$ where $[n:m]$ is the least common multiple between $n$ and $m$. 
The construction of the fiberization isometry and the shift-modulation range function can be done for general $n$ and $m$. For the sake of simplicity we will fix $m=15$ and $n=12$.

Since $(12:15)=3$ and $[12:15]=60$, we have that $E=\{0, \frac1{3}, \frac{2}{3}\}$ and 
$F+\Dd^*=\frac1{60}\Z_{60}$. Then, $E^*=3\Z$ and we can fix $\Pi_{E^*/\Dd}$ as $\{1, 3, 6, 9\}$ and 
$\Pi_{\G/E^*}$ as $\{0, 1, 2\}$.
Now, in order to construct $\Pi_{G/\Dd^*}$ following equation (\ref{sec-H}) we choose $\Pi_{G/F+\Dd^*}=[0, \frac1{60})$ and $\Pi_{F/E}=\{0, \frac{3}{15}, \frac{6}{15}, \frac{9}{15}, \frac{12}{15}\}$. 
Then, the section  $\Pi_{G/\Dd^*}$ is $\Pi_{G/\Dd^*}=[0, \frac1{60})\cup [\frac{12}{60}, \frac{13}{60})\cup[\frac{24}{60}, \frac{25}{60})\cup[\frac{36}{60}, \frac{37}{60})\cup[\frac{48}{60}, \frac{49}{60})$.

For $\p\in L^2(\Tor)$ the fiberization isometry applied to $\p$ is 
$$\T\p(x, \xi)=(Z\p(x,\xi), Z\p(x,\xi+3), Z\p(x,\xi+ 6), Z\p(x,\xi+9))$$
where $x\in \Pi_{G/\Dd^*}$, $\xi\in \Pi_{\G/E^*}=\{0, 1, 2\}$ and the Zak transform  is given by 
$Z\p(x, \xi)=\sum_{j=0}^{11}\p(x-\frac{j}{12})e^{2\pi i\frac{j}{12}\xi}$.

We now focus in the particular case when $\p=\chi_{[0, \frac1{60})}$.  We have that for all $\xi\in \{0, 1, 2\}$, $Z\p(x, \xi)=1$ if $x\in [0, \frac1{60})$ and  $Z\p(x, \xi)=0$ if $x\neq [0, \frac1{60})$. Moreover, 
it can be proven that for each $r\in\{0,3 ,6, 9, 12\}$, $Z(T_{\frac{r}{15}}\p)(x, \xi)=1$ if $x\in[\frac{4r}{60}, \frac{4r+1}{60})$
and $Z(T_{\frac{r}{15}}\p)(x, \xi)=0$ if $x\in \Pi_{G/\Dd^*}\setminus[\frac{4r}{60}, \frac{4r+1}{60})$.

Then, for all $\xi\in\{0, 1, 2\}$ and for each $r\in\{0,3 ,6, 9, 12\}$
\begin{equation}\label{tau-example}
\T(T_{\frac{r}{15}}\p)(x, \xi)=
\begin{cases}
(1,1,1,1)&\textrm{ if } x\in  [\frac{4r}{60}, \frac{4r+1}{60})\\
(0,0,0,0)&\textrm{ if } x\in \Pi_{G/\Dd^*}\setminus  [\frac{4r}{60}, \frac{4r+1}{60}).
\end{cases}
\end{equation}

The shift-modulation range function associated to $S_{(F, \Dd)}(\p)$ is 
$$J(x, \xi)=\Span\{ \T (T_{\frac{r}{15}}\p)(x,\xi):\, r\in\{0,3 ,6, 9, 12\}\} ,$$ 
and, using (\ref{tau-example}) we obtain
$J(x, \xi)=\Span\{(1,1,1,1)\},$ for $(x,\xi)\in [0,\frac1{60})\times\{0, 1, 2\}$.

Consider now the translation of $\p=\chi_{[0, \frac1{60})}$ by the element $\frac1{2}$. We will show that 
$T_{\frac1{2}}\p\notin S_{(F, \Dd)}(\p)$.
With similar computations as above, it can be shown that, for all $\xi\in\{0, 1, 2\}$
\begin{equation*}
\T(T_{\frac{1}{2}}\p)(x, \xi)=
\begin{cases}
(1+e^{\pi i\xi},1-e^{\pi i\xi},1+e^{\pi i\xi},1-e^{\pi i\xi})&\textrm{if } x\in  [0, \frac{1}{60})\\
(0,0,0,0)&\textrm{if } x\in \Pi_{G/\Dd^*}\setminus  [0, \frac{1}{60}).
\end{cases}
\end{equation*}
From this, we have that, for $x\in[0,\frac1{60})$,  $\T(T_{\frac{1}{2}}\p)(x, 1)=(0,2,0,2)\notin J(x, 1)$.
Using Theorem \ref{thm-ppal} we can conclude that $T_{\frac1{2}}\p\notin S_{(F, \Dd)}(\p)$.
\begin{flushright}
$\square$
\end{flushright}
\end{example}

\section*{Acknowledgements}
We want to thank the referees for their suggestions that led to a great improvement of the manuscript. We also thank to Magal\'i Anastasio for fruitful discussions.

\thispagestyle{empty}

\normalsize
\baselineskip=17pt

\end{document}